\documentclass[12pt,twoside]{amsart}
\usepackage{amssymb}
\usepackage[all]{xy}

\nonstopmode

\numberwithin{equation}{section} 

\font\tengothic=eufm10 scaled\magstep 1 \font\sevengothic=eufm7
scaled\magstep 1
\newfam\gothicfam
      \textfont\gothicfam=\tengothic
      \scriptfont\gothicfam=\sevengothic

\newtheorem{theorem}{Theorem}[section]

\newtheorem{proposition}[theorem]{Proposition}

\newtheorem{conjecture}[theorem]{Conjecture}

\theoremstyle{definition}
\newtheorem{definition}[theorem]{Definition} 
\newtheorem{remark}[theorem]{Remark}
\newtheorem{problem}[theorem]{Problem}

\newtheorem{notation}[theorem]{Notation}

\newcommand{\Hom}{\operatorname{Hom}}
\newcommand{\Ext}{\operatorname{Ext}}
\newcommand{\rank}{\operatorname{rank}}

\newcommand{\cO}{{\mathcal O}}

\newcommand{\cT}{{\mathcal T}}

\newcommand {\RR}{\mathbb{R}}

\newcommand {\ZZ}{\mathbb{Z}}
\newcommand {\NN}{\mathbb{N}}
\newcommand {\PP}{\mathbb{P}}

\begin{document}
\title[Derived category of toric varieties]
{Frobenius splitting and Derived category of toric varieties}

\author[L.\ Costa, R.M.\ Mir\'o-Roig]{L.\ Costa$^*$, R.M.\
Mir\'o-Roig$^{**}$}

\address{Facultat de Matem\`atiques,
Departament d'Algebra i Geometria, Gran Via de les Corts Catalanes
585, 08007 Barcelona, SPAIN } \email{costa@ub.edu}

\address{Facultat de Matem\`atiques,
Departament d'Algebra i Geometria, Gran Via de les Corts Catalanes
585, 08007 Barcelona, SPAIN } \email{miro@ub.edu}

\date{\today}
\thanks{$^*$ Partially supported by MTM2007-61104.}
\thanks{$^{**}$ Partially supported by MTM2007-61104 .}



\begin{abstract} In this paper, we will use the splitting of the Frobenius direct image of line bundles on
toric varieties to explicitly construct an orthogonal basis of  line bundles in the derived
category $D^b(X)$ where $X$ is a Fano toric variety with (almost) maximal Picard number.

\end{abstract}


\maketitle

\tableofcontents


 \section{Introduction} \label{intro}

Let $Y$ be a smooth projective variety defined over an
algebraically closed field $K$ of characteristic zero and let
$D^b(Y)=D^b({\cO}_Y$-$mod)$ be the derived category of bounded
complexes of coherent sheaves of ${\cO}_Y$-modules. $D^b(Y)$ is
one of the most important algebraic invariants of a smooth
projective variety $Y$ and we would like to know whether $D^b(Y)$ is freely and finitely
generated or, more precisely, whether there exists a full strongly exceptional collection of coherent sheaves on $Y$. In spite of the increasing interest in
understanding the structure of $D^b(Y)$, very little progress has
been achieved. The existence of a full strongly
exceptional collection  of coherent sheaves
on a smooth projective variety $Y$ is very restrictive for $Y$, e.g. the Grothendieck group
$K_0(Y)=K_0(\cO _Y-mod)$ has to be a finitely generated abelian group.
There exists a nice class of algebraic varieties, the class of smooth projective toric varieties, satisfying this condition on the Grothendieck group and King \cite{Ki} conjectured

\begin{conjecture} \label{conjeturaking} Every smooth
complete  toric variety has a full strongly exceptional
collection of line bundles.
\end{conjecture}

 There are a lot of contributions to the above
conjecture. For instance, it turns out to be true for projective
spaces \cite{Be}, multiprojective spaces (\cite{CMZ}; Proposition
4.16), smooth complete toric varieties with Picard number $\le 2$
(\cite{CMZ}; Corollary 4.13) and smooth complete toric varieties
with a splitting fan (\cite{CMZ}; Theorem 4.12). Nevertheless some
restrictions are required because, recently, in \cite{HP}, Hille
and Perling constructed an example of smooth non Fano toric
surface which does not have a full strongly exceptional collection
made up of line bundles. It is quite natural to conjecture

\begin{conjecture}\label{conj1} Every smooth
complete Fano toric variety has a full strongly exceptional
collection of line bundles
\end{conjecture}

There are some numerical evidences towards the above conjecture (see, for instance, \cite{CM}).
So far only partial results are known and we want to point out
that the hypothesis Fano is not necessary. In fact, in \cite{CMZ};
Theorem 4.12, we constructed full strongly exceptional collections
of line bundles on families of smooth complete toric varieties
none of which is entirely of Fano varieties.

\vspace{3mm}
The goal of this paper is to investigate the structure of $D^b(X)$
where $X$ is a smooth Fano toric variety with  (almost) maximal Picard number
and to prove that for such kind of varieties always exists a full strongly exceptional collection of line bundles (see Theorem \ref{mainthm}). Hence, our main
result provides new evidences towards Conjecture \ref{conj1}.
In order to get a good
candidate to be a full strongly exceptional collection of line
bundles and to achieve our main result we use, as a main tool, the splitting of
the Frobenius direct image of line bundles on smooth complete toric varieties.
This approach will give us a full collection of line bundles
on $X$  and, in the last part of the work, we will apply Bondal's criterium (see Proposition \ref{Bondalcriteri})  to conclude that such collection  can be ordered  in such a way that we get a full strongly exceptional collection on $X$.

\vspace{3mm}

Next we outline the structure of this paper. In section 2, we fix the notation and we summarize the basic facts on toric varieties needed in the sequel. In particular, we
recall the classification of smooth Fano toric varieties with (almost) maximal Picard number and we explicitly describe the splitting of the Frobenius image of line bundles on toric images. Section 3 contains the main result of this work. We first
 briefly review the notions of exceptional sheaves, exceptional
collections of sheaves and strongly exceptional collections of
sheaves as well as the facts on derived categories needed later.
At the end, we prove the existence of an orthogonal basis in $D^b(X)$ made up of lines bundles, where $X$ is an $n$-dimensional  smooth Fano toric variety with Picard number $2n-1\le \rho (X) \le 2n$ if $n$ is even; and
$\rho (X)=2n-1$ if $n$ is odd (see Theorem \ref{mainthm}).


\section{Toric varieties and Frobenius splitting}

\vspace{3mm}  In this section we  deal with
 $d$-dimensional toric varieties $X$ with (almost) maximal Picard
number. We first recall their classification (Theorem \ref{classificacioPolitop} and Proposition
\ref{classificacioVarietat}) and we use it to explicitly
describe the splitting of the direct image
$(\pi_p)_*(\cO_X)$ where $\pi_p$ is the Frobenius morphism.  To start with, we fix the notation
and we recall the facts on toric varieties that we will use along this
paper refereing to \cite{Fu} and \cite{Oda} for more details.

\vskip 2mm Let $Y$ be a smooth complete toric variety of dimension
$n$ over an algebraically closed field $K$ of characteristic zero
characterized by a fan $\Sigma:=\Sigma(Y)$ of strongly convex
polyhedral cones in $N\otimes _{\ZZ} \RR$ where $N$ is the lattice
$\ZZ^n$, i.e. $N$ is a free abelian group of rank $n$ and we will
denote by $e_0$, $\ldots $, $e_{n-1}$ a $\ZZ$-basis of $N$. Let
$M:=\Hom_{\ZZ}(N,\ZZ)$ denote the dual lattice and $\hat{e}_0$,
$\ldots $, $\hat{e}_{n-1}$ the dual basis of $e_0$, $\ldots $,
$e_{n-1}$. If $\sigma$ is a cone in $N$, the dual cone $\sigma
\check{}$ is the set of vectors in $M$ that are nonnegative in
$\sigma$. This determines a commutative semigroup $\sigma \check{}
\cap M$ and we set \[ U_{\sigma}=Spec(K[S_{\sigma}]) \] to denote
the open affine toric subvariety.

 For any $0\le i \le n$, we put $\Sigma (i):=\{
\sigma \in \Sigma \mid \text{ dim}(\sigma ) =i   \}$. In
particular, to any 1-dimensional cone $\sigma \in \Sigma(1)$ there
is a unique generator $v \in N$, called {\em ray generator}, such
that $\sigma \cap N=\ZZ_{\ge 0}\cdot v$. We label the set of
generators in $N$ of the 1-dimensional cones by $\{v_{i} \mid i\in
J\}$. There is a one-to-one correspondence between such ray
generators $\{v_{i} \mid i\in J\}$ and simple toric divisors $\{Z_{i}
\mid i\in J\}$ on $Y$. The following notion is due to
 V.V. Batyrev (see \cite{Bat}).

\vspace{3mm}

\begin{definition}\label{primitrelat}
 Let $Y$ be a smooth toric variety. A set of toric divisors $\{Z_1,...,Z_k\}$ on $Y$ is called a
 {\em primitive set} if $Z_1\cap \cdot \cdot \cdot \cap Z_k=\emptyset $
 but $Z_1\cap \cdot \cdot \cdot \cap \widehat{ Z_j}
 \cap \cdot \cdot \cdot \cap Z_k\ne \emptyset$ for all $j$, $1\le j \le k$.
 Equivalently, this means $<v_1,...,v_k>\notin \Sigma$
 but $<v_1,...,\widehat{v_j},...,v_k>\in \Sigma$ for all $j$ and we call
  $P=\{ v_1,...,v_k \}$ a {\em primitive collection}.

 If $S:=\{Z_1,...,Z_k\}
 $ is a primitive set, the element $v:=v_1+...+v_k$ lies
 in the relative interior of a unique cone of $\Sigma $, say the cone
 generated by $v_1',...,v_s'$ and
 $v_1+...+v_k=a_1v_1'+...+a_sv_s'$ with $a_i>0$
 is the corresponding {\em primitive relation}.
\end{definition}

\vspace{3mm}

 If $Y$ is a smooth toric variety of dimension $n$
(hence $n$ is also the dimension of the lattice $N$) and $m$ is
the number of toric divisors of $Y$ (and hence the number of
1-dimensional rays in $\Sigma $) then the Picard number of $Y$ is
$\rho (Y)=m-n$ and the anticanonical divisor $-K_Y$ is given by
$-K_Y=Z_1+\cdots+Z_m$. A smooth toric Fano variety $Y$ is a smooth
toric variety with the anticanonical divisor $-K_Y$ ample.

\vskip 2mm

It is well known that isomorphism classes of $d$-dimensional
smooth Fano toric varieties  correspond to isomorphism classes of
smooth Fano $d$-polytopes, that is, fully dimensional convex
lattice polytopes in $\RR^d$ such that the origin is in the
interior of the polytopes and the vertices of every facet is a
basis of the integral lattice $\ZZ^d \subset \RR^d$. Smooth Fano
$d$-polytopes have been intensively studied during the last
decades and completely classified up to dimension 4 (\cite{Bat} and \cite{Sat}). In higher dimension, they are classified under some additional assumptions; for instance, when the polytopes have few vertices (see \cite{Kl}), maximal number of vertices (see \cite{Casagrande} and \cite{Obro}) or some extra symmetries (see \cite{CC}).

\vskip 2mm
In our works \cite{CM} and \cite{CMZ}; we described the bounded derived category of smooth Fano $d$-dimensional polytopes with few vertices (see \cite{CMZ}; Corollary 4.13) and, in this paper, we will deal with smooth Fano
$d$-dimensional polytopes with maximal number of  vertices.
It is known that $3d $ is an upper bound for the
number of vertices of a Fano $d$-polytope and in the following
theorem we recall the classification of  smooth Fano
$d$-polytopes with maximal, if $d$ is odd,  and (almost) maximal, if $d$ is even, number of vertices. This classification
 turns out to be the classification of smooth Fano
$d$-dimensional toric varieties with maximal, if $d$ is odd, and (almost) maximal, if $d$ is even,
Picard number.

\vspace{3mm}
\begin{theorem}
\label{classificacioPolitop} Let $P \subset N_{\RR}$ be a  smooth Fano polytope and $e_0, \cdots, e_{d-1}$ a basis
of the lattice $\ZZ^d$. The following holds:

\vskip 2mm
\begin{itemize}
\item[(1)]
The number of vertices of $P$ is bounded by $3d$ if $d$ is even
and by $3d-1$ if $d$ is odd.
\item[(2)]
If $d$ is even and $P$ has exactly $3d$ vertices, then $P$ is
the convex hull of the $3d$ points
\[\pm e_0, \quad  \pm e_1, \quad \cdots, \quad \pm e_{d-2}, \quad \pm e_{d-1} \]
\[ \pm (e_0-e_1),  \quad \pm(e_2-e_3), \quad  \cdots, \quad \pm (e_{d-2}-e_{d-1}). \]
\item[(3)]
 If $d$ is even and $P$ has exactly $3d-1$ vertices, then $P$ is
the convex hull of the $3d-1$ points
\[ e_0, \quad \pm e_1, \quad \cdots, \quad \pm e_{d-2}, \quad \pm e_{d-1} \]
\[ \pm (e_0-e_1), \quad \pm(e_{2}-e_3), \quad  \cdots, \quad \pm (e_{d-2}-e_{d-1}). \]

\item[(4)] If $d$ is odd  and $P$ has $3d-1$ vertices, then $P$ is the
convex hull of the $(3d-1)$ points
\[e_0, \quad \pm e_1, \cdots, \quad \pm e_{d-1} \]
\[ e_1-e_0, \quad \pm(e_1-e_2), \quad \pm (e_3-e_4), \cdots, \quad \pm (e_{d-2}-e_{d-1}) \]
or the convex hull of the $(3d-1)$  points
\[\pm e_0, \quad \pm e_1, \cdots, \quad \pm e_{d-1} \]
\[  \pm(e_1-e_2), \quad \pm (e_3-e_4), \cdots, \quad \pm (e_{d-2}-e_{d-1}) \]
\end{itemize}
\end{theorem}
\begin{proof}
See \cite{Obro}; Theorem 1 and \cite{Casagrande}; Theorem 1.
\end{proof}

\vspace{3mm}

Recall that the Picard number of a $d$-dimensional smooth Fano toric
variety is equal to the number of vertices of the associated Fano
polytope minus $d$. So, if we denote by $S_2$ the blow up of $\PP^2$
at two torus-invariant points and by $S_3$  the blow up of
$\PP^2$ at three torus-invariant points, the above classifying
result can be read off in the following way

\vspace{3mm}

\begin{proposition}
\label{classificacioVarietat} Let $X$ be a $d$-dimensional
smooth Fano toric  variety with Picard number $\rho_X$. Then,
\vskip 2mm \begin{itemize}
\item[(1)]
If $d$ is even, $\rho_X \leq 2d$ and there is up to isomorphism
only one $X$ with $\rho_X=2d$, namely $(S_3)^{\frac{d}{2}}$, and
one with $\rho_X=2d-1$, namely $S_2 \times (S_3)^{\frac{d-2}{2}}$.
\item[(2)]
2) If $d$ is odd, $\rho_X \leq 2d-1$ and there are up to
isomorphism precisely two $X$ with $\rho_X=2d-1$, namely $\PP^1
\times (S_3)^{\frac{d-1}{2}}$ or a unique determined toric
$(S_3)^{\frac{d-1}{2}}$-fiber bundle over $\PP^1$.
\end{itemize}
\end{proposition}
\begin{proof}
See \cite{Obro} and \cite{Nill}; Proposition 4.1.
\end{proof}

\vspace{3mm}

The main goal of the next section is to give  an orthogonal basis made up of line bundles
for the derived category $D^b(X)$ of bounded complexes of coherent sheaves
on the toric varieties $X$ described in Proposition \ref{classificacioVarietat}, mainly on smooth Fano toric varieties of dimension $d$ with (almost)
maximal Picard number $\rho _X$. If $d$ is even and $2d-1\le \rho _X \le 2d$ or $d$ is odd, $\rho _X=2d-1$ and  $X$
isomorphic to  $\PP^1 \times  (S_3)^{\frac{d-1}{2}}$ then, applying  \cite{CMZ}; Theorem 4.17, we will see that  there is an
orthogonal basis for the derived category $D^b(X)$ of bounded complexes of coherent sheaves on
$X$ made up of line bundles. For the remaining case, namely
a toric $(S_3)^{\frac{d-1}{2}}$-fiber bundle over $\PP^1$ we will
explicitly compute such basis. To this end, we need to fix some
notation and to develop some technical results.

\vspace{3mm}

From now on, for any odd integer $d \geq 3$, we will denote by
$X_d$ the toric $(S_3)^{\frac{d-1}{2}}$-fiber bundle over $\PP^1$
quoted in Proposition \ref{classificacioVarietat}.

\vspace{3mm}

For any smooth projective toric variety $X$, we denote by $P_X(t)$ its
Poincar\'e polynomial. It is well known that the topological Euler
characteristic of $X$, $\chi(X)$ verifies
\[ \chi(X)=P_X(-1)\]
and $\chi(X)$ coincides with the number of maximal cones of $X$,
that is, with the rank of the Grothendieck group $K_0(X)$ of $X$.
On the other hand, since $X_d$ is a $(S_3)^{\frac{d-1}{2}}$-fiber bundle over $\PP^1$ we
have (\cite{Fu};Pag 92-93):

\[ P_{X_d}(t)=P_{(S_3)^{\frac{d-1}{2}}}(t) P_{\PP^1}(t).\]

Thus putting altogether we deduce

\begin{equation} \label{rankK0} \rank(K_0(X_d))= P_{(S_3)^{\frac{d-1}{2}}}(-1) P_{\PP^1}(-1)=2 \cdot 6^{\frac{d-1}{2}}. \end{equation}

By Theorem \ref{classificacioPolitop} and Proposition
\ref{classificacioVarietat}, $X_d$ is the toric variety associated
to the convex hull of the $(3d-1)$ points

\[e_0, \quad \pm e_1, \cdots, \quad \pm e_{d-1} \]
\[ e_1-e_0, \quad \pm(e_1-e_2), \quad \pm (e_3-e_4), \cdots, \quad \pm (e_{d-2}-e_{d-1}) \]

 $e_0, \cdots, e_{d-1}$ being a basis of the lattice $\ZZ^d$.
Denote by

 \[v_0=e_0,  \quad v_{2k-1}=e_k, \quad  v_{2k}=-e_k, \quad \mbox{for } 1 \leq k \leq d-1=2l \]
 \[w_0=e_1-e_0,  \quad w_{2j-1}=e_{2j-1}-e_{2j}, \quad  w_{2j}=e_{2j}-e_{2j-1}, \quad \mbox{for } 1 \leq j \leq l, \]

 \noindent the ray generators of the fan $\Sigma_d$ associated to $X_d$. For
 a later use, it is convenient to remark that the following is the
 list of all primitive collections on $X_d$, $d=2l+1$ (see \cite{CC}; Section 2)

 \begin{equation}
 \label{primitive}
 \begin{array}{lll}
 \{v_{2k-1},v_{2k} \} & \mbox{for} & 1 \leq k \leq 2l, \\
  \{w_{2j-1},w_{2j} \} & \mbox{for} & 1 \leq j \leq l, \\
   \{w_{2j-1},v_{4j-2} \} & \mbox{for} & 1 \leq j \leq l, \\
      \{w_{2j-1},v_{4j-1} \} & \mbox{for} & 1 \leq j \leq l, \\
         \{w_{2j},v_{4j-3} \} & \mbox{for} & 1 \leq j \leq l, \\
            \{w_{2j},v_{4j} \} & \mbox{for} & 1 \leq j \leq l, \\
             \{w_{0},v_{0} \} . \\
            \end{array}
 \end{equation}

\vspace{3mm}

For the rest of the work, we will use the following notation when
we will deal with toric divisors on $X_d$, for $d=2l+1 \geq 3$. We
will denote by
\vskip 2mm
\begin{itemize}
\item $Z_{i}^+$ the toric divisor associated to $e_i$, $0 \leq i
\leq d-1$, \item $Z_{i}^-$ the toric divisor associated to $-e_i$,
$1 \leq i \leq d-1$, \item $D_0$ the toric divisor associated to
$e_1-e_0$, \item $D_{j}^+$ the toric divisor associated to
$e_{2j-1}-e_{2j}$, $1 \leq j \leq l$,  \item $D_{j}^-$
the toric divisor associated to $-(e_{2j-1}-e_{2j})$, $1 \leq j
\leq l$.
\end{itemize}

\vspace{3mm}

Given any smooth complete toric variety $Y$, Bondal described a method to produce a
candidate collection of line bundles on $Y$, which for certain
classes  of Fano toric varieties is expected to be an orthogonal basis of the derived category $D^b(Y)$ of bounded complexes of coherent sheaves on $Y$. This method requires to compute
 the
different summands appearing on the Frobenius splitting of the
tautological line bundle which will be achieved applying the algorithm that we will describe now.

\vskip 2mm For any smooth complete toric variety $Y$ of
dimension $n$ with an $n$-dimensional torus $T$ acting on it and
for any integer $\ell \in \ZZ$, there is a well-defined toric
morphism
$$\pi_{\ell }:Y\longrightarrow Y$$ which restricts, on the torus
$T$, to the map $$\pi_{\ell }:T \longrightarrow T, \quad t\mapsto
t^{\ell }.$$ The map $\pi_{\ell } $ is the factorization map with
respect to the action of the group of $\ell $ torsion of $T$. We
fix a prime integer $p\gg 0$. By  \cite{JF}; Theorem 1 and
Proposition 2, $(\pi _{p })_{*}(\cO_Y)^{\vee}$ is a vector bundle of rank
$p^n$ which splits into a sum of line bundles $$ (\pi _{p
})_{*}(\cO_Y)^{\vee} =\oplus _{\chi } \cO_Y(D_{\chi})$$ where the sum is
taken over the group of characters of the $p$-torsion subgroup of
$T$. Moreover,
$$c_1((\pi _{p })_{*}(\cO_Y)^{\vee})=\cO_Y(-\frac{p^{n-1}(p-1)}{2}K_Y)$$
where $K_Y$ is the canonical divisor of $Y$.

\vspace{3mm} For sake of completeness, we recall here the
algorithm described  by Thomsen in \cite{JF} that we will apply
later in order to get explicitly the summands of the splitting of
$ (\pi _{p })_{*}(\cO_{X_d})$.

\vspace{3mm}

 Given any  smooth complete toric variety $Y$ of
dimension $n$, Picard number $\rho$ (hence $n+\rho$ toric
divisors) and Group of Grothendieck $K_0(Y)$ of rank $s$ (hence
$s$ maximal cones),  we consider $\{\sigma_1, \cdots,
\sigma_{s}\}$ the set of maximal cones of the fan $\Sigma$
associated to $Y$ and we denote by $v_{i_1}, \cdots, v_{i_n}$ the
generators of $\sigma_i$. Recall that since $Y$ is smooth, every
rational cone $\sigma \in \Sigma$ is generated by a part of a
$\ZZ$-basis of $N$. For each index $i$, $1 \leq i \leq s$, we form
the matrix $A_i \in GL_n(\ZZ)$ having as the $j$-th row the
coordinates of $v_{i_j}$ expressed in the basis $e_1, \cdots, e_n$
of $N$. Let $B_i=A_i^{-1} \in GL_n(\ZZ)$ and we denote by $w_{ij}$
the $j$-th column vector in $B_i$. Introducing the symbols
$Y^{\hat{e_1}}, \cdots, Y^{\hat{e_n}}$, we form the ring

\[R=K[(Y^{\hat{e_1}})^{\pm 1}, \ldots, (Y^{\hat{e_n}})^{\pm 1}]  \]
which is the coordinate ring of the torus $T \subset Y$ and for
any $i$,  $1 \leq i \leq s $, the coordinate ring of the open
affine subvariety $U_{\sigma_i}$ of $Y$ corresponding to the cone
$\sigma_i$ is the subring

\[ R_i=K[Y^{w_{i1}}, \cdots, Y^{w_{in}}] \subset R\]
where we use the notation

\[ Y^{w}:=(Y^{\hat{e_1}})^{w_1} \cdots
(Y^{\hat{e_n}})^{w_n}\] if $w=(w_1, \cdots,w_n)$. For simplicity
we will also write $Y_{ij}:=Y^{w_{ij}}$.

For each $i$ and $j$, we denote by $R_{ij}$ the coordinate ring of
$\sigma_i \cap \sigma_j$ and we define
\begin{equation} \label{Iij}
I_{ij}:= \{v \in M_{n \times 1}(\ZZ) | Y_i^v \quad \mbox{is a unit
in } R_{ij} \},
\end{equation}
\begin{equation} \label{Cij}
C_{ij}:= B_j^{-1}B_i \in GL_n(\ZZ)
\end{equation}
where we use the notation $Y_i^v:= (Y_{i1})^{v_1} \cdots
(Y_{in})^{v_n}$ being $v$ a column vector with entries $v_1,
\cdots, v_n$.

For every $p \in \NN$ and $w \in I_{ij}$, we define
\[ P_p^n:= \{v \in M_{n \times 1}(\ZZ) | 0 \leq v_i < p \}\]
and the maps
\[h_{ijp}^{w}: P_p ^n\rightarrow R_{ij} \]
\[r_{ijp}^{w}: P_p ^n\rightarrow P_{p} ^n \]
by means of the following equality: for any $v \in P_p$
\[C_{ij}v + w= p \cdot h_{ijp}^{w}(v)+ r_{ijp}^{w}(v). \]
By \cite{JF}; Lemma 2 and Lemma 3, these maps exist and they are
unique.

\vskip 2mm Recall that any toric Cartier divisor $D$ on $Y$ can be
represented in the form $\{(U_{\sigma_i}, Y_i^{u_i}) \}_{\sigma_i
\in \Sigma}$, $u_i \in M_{n\times 1}(\ZZ)$ (see \cite{Fu}; Chapter 3.3).
Once fixed the set $\{(U_{\sigma_i}, Y_i^{u_i}) \}_{\sigma_i \in
\Sigma}$ which represents a toric Cartier divisor $D$, we define
\[ u_{ij}=u_j-C_{ij}u_i.\]
Notice that if $\cO_{Y}(D)=\cO_Y$ is the trivial line bundle, then
for any pair $i$, $j$, we have $u_{ij}=0$.

For any $p \in \ZZ$ and any toric Cartier divisor $D$ on $Y$,
$(\pi_{p })_*(\cO_Y(D))^{\vee}$ is defined as follows: we fix a set
$\{(U_{\sigma_i}, Y_i^{u_i}) \}_{\sigma_i \in \Sigma}$
representing $D$ and we choose an index $l$ of a cone $\sigma_l
\in \Sigma$. Let $D_v$, $v \in P_p^n$, denote the Cartier divisor
represented by the set  $\{(U_{\sigma_i}, Y_i^{h_i}) \}_{\sigma_i
\in \Sigma}$ where, by definition
\[h_i=h_i^v:=h_{lip}^{u_{li}}(v).\]
Then, we have \begin{equation} \label{sumands} ( \pi_{p
})_*(\cO_Y(D))^{\vee}= \bigoplus_{v \in P_p^n} \cO_Y(D_v).\end{equation}

\begin{remark}
\label{explicacio} Recall that if $h_i=(h_{i1}, \cdots, h_{in})$
and $\alpha_{i1}^j, \cdots, \alpha_{in}^j$ are the entries of the
$j$-th column vector of $B_i$, then by definition
\[ Y_i^{h_i}=({Y^{\hat{e_1}}}^{\alpha_{i1}^1}\cdots {Y^{\hat{e_n}}}^{\alpha_{in}^1})^{h_{i1}}
({Y^{\hat{e_1}}}^{\alpha_{i1}^2}\cdots
{Y^{\hat{e_n}}}^{\alpha_{in}^2})^{h_{i2}} \cdots
({Y^{\hat{e_1}}}^{\alpha_{i1}^n}\cdots
{Y^{\hat{e_n}}}^{\alpha_{in}^n})^{h_{in}}. \] We denote by
\[ l_{\sigma_i}=(\alpha_{i1}^1 h_{i1}+\alpha_{i1}^2 h_{i2}+\cdots+ \alpha_{i1}^n
h_{in})\hat{e_1}+ (\alpha_{i2}^1 h_{i1}+\alpha_{i2}^2
h_{i2}+\cdots+ \alpha_{i2}^n h_{in})\hat{e_2}+ \cdots \] \[ \cdots
+ (\alpha_{in}^1 h_{i1}+\alpha_{in}^2 h_{i2}+\cdots+ \alpha_{in}^n
h_{in})\hat{e_n} \in M. \] According to this notation, if $D_v$ is
the Cartier divisor represented by the set  $\{(U_{\sigma_i},
Y_i^{h_i}) \}$, then \[D_v= \beta_v^1Z_1+ \cdots+
\beta_v^{n+\rho}Z_{n+\rho}
\] where
$$\beta_v^j=-l_{\sigma_k}(v_j)$$ for any maximal cone $\sigma_k$
containing the ray generator $v_j$ associated to the toric divisor
$Z_j$. Indeed, for any pair of maximal cones $\sigma_k$ and
$\sigma_{k'}$ containing $v_j$,
$l_{\sigma_k}(v_j)=l_{\sigma_{k'}}(v_j)$.
\end{remark}

Using this algorithm we are going to prove

\begin{proposition} \label{summandspliting} Let $X_d$ be the toric $(S_3)^{\frac{d-1}{2}}$-fiber bundle over
$\PP^1$, $d=2l+1$ and $p  \gg 0$ a prime integer. With the above
notations the different summands of $(\pi_p)_*(\cO_{X_d})^{\vee}$ are
{\footnotesize\[
 \cT_3 \otimes \bigotimes_{k=2}^{l}(\cO
\oplus \cO(Z_{2k}^-+D_k^+) \oplus \cO(Z_{2k-1}^-+D_k^-) \oplus
\cO(Z_{2k-1}^-+Z_{2k}^-) \oplus \cO(Z_{2k-1}^-+Z_{2k}^-+D_k^-)
\oplus \cO(Z_{2k-1}^-+Z_{2k}^-+D_k^+))
\]}
where {\footnotesize\[ \begin{array}{ll} \cT_3 \cong & \cO \oplus
\cO(D_0) \oplus \cO(Z_1^-+D_1^-) \oplus \cO(Z_1^-+D_1^-+D_0)
\oplus \cO(Z_2^-+D_1^+) \oplus \cO(Z_2^-+D_1^++D_0) \\ &  \oplus
\cO(Z_1^-+Z_2^-) \oplus \cO(Z_1^-+Z_2^-+D_1^+)  \oplus
\cO(Z_1^-+Z_2^-+D_1^-) \oplus \cO(Z_1^-+Z_2^-+D_0) \\ & \oplus
\cO(Z_1^-+Z_2^-+D_1^++D_0)  \oplus \cO(Z_1^-+Z_2^-+D_1^-+D_0).
\end{array} \] }
\end{proposition}
\begin{proof}
By  \cite{JF}; Theorem 1 and Proposition 2, $(\pi _{p
})_{*}(\cO_{X_d})^{\vee}$ is a vector bundle of rank $p^d$ which splits
into a sum of line bundles \begin{equation} \label{split}  (\pi
_{p })_{*}(\cO_{X_d})^{\vee} =\bigoplus_{v^d \in P_p^d}
\cO_X(D_{v^d})\end{equation}
 and using the
above algorithm, we will determine all these different summands
$\cO_{X_d}(D_{v^d})$ moving $v^d \in P_p^d$. To this end, we will
proceed by induction on odd $d$.

\vspace{3mm} Assume $d=3$. Take $e_0, e_1, e_{2}$ be a $\ZZ$-basis
of the lattice $\ZZ^3$ and denote by
 \[v_0=e_0,  \quad v_{1}=e_1, \quad v_{2}=-e_1, \quad v_{3}=e_2, \quad  v_{4}=-e_2 \]
 \[w_0=e_1-e_0,  \quad w_{1}=e_{1}-e_{2}, \quad  w_{2}=e_{2}-e_{1} \]
 the ray generators of the fan $\Sigma_3$ associated to $X_3$.

 It follows from Remark \ref{explicacio} that in order to get all the
 different summands appearing in the splitting (\ref{split}),
 it is enough to determine $l_{\sigma_1}$, $l_{\sigma_2}$ and
$l_{\sigma_3}$ where $\sigma_1$, $\sigma_2$ and $\sigma_3$ are
three maximal cones of $\Sigma_3$ involving all the
ray generators $v_i$, $0 \leq i \leq 4$ and $w_i$, $0 \leq i \leq
2$. We choose the following three maximal cones of $\Sigma_3$:
\[  \sigma_1:= \langle v_0,v_1,v_3 \rangle,  \quad   \sigma_2:= \langle v_2, w_0, w_2,  \rangle,
  \quad \sigma_3:= \langle v_0,v_4,  w_{1} \rangle.  \]

The matrices $A_i$, $1 \leq i \leq 3$, having as the $j$-th row
the coordinates of the $j$-vector of $\sigma_i$ expressed in the
basis $e_0, e_1, e_{2}$ are:

\[A_1= \left ( \begin{array}{ccc}  1 &0& 0 \\ 0& 1 & 0  \\ 0 & 0
& 1 \\ \end{array} \right ) \quad A_2= \left (
\begin{array}{ccc}  0 &-1& 0 \\-1& 1 & 0  \\ 0 & -1 & 1 \\
\end{array} \right ) \quad A_3= \left ( \begin{array}{ccc}  1 &0& 0 \\ 0& 0 & -1  \\ 0 & 1
& -1 \\ \end{array} \right ) \] and their inverses are given by

\[B_1= \left ( \begin{array}{ccc}  1 &0& 0 \\ 0& 1 & 0  \\ 0 & 0
& 1 \\ \end{array} \right ) \quad B_2= \left (
\begin{array}{ccc}  -1 &-1& 0 \\-1& 0 & 0  \\ -1 & 0 & 1 \\
\end{array} \right ) \quad B_3= \left ( \begin{array}{ccc}  1 &0& 0 \\ 0& -1 & 1  \\ 0 & -1
& 0 \\ \end{array} \right ). \]
 Fix the index $l=1$ corresponding to the cone $\sigma_1$. By (\ref{Cij}),
\[C_{1i}=(B_i)^{-1} B_1=A_i \]
and, as we pointed out before, if $\{U_{\sigma_j}, X_j^{u_j}
\}_{\sigma_j \in \Sigma_3}$ represents the zero divisor then for
any pair $i$, $j$,
\[u_{1i}=u_i-C_{1i}u_1=0. \]
Hence, for any $v \in P_p^3$, $h_i^{v}:=h_{1ip}^{u_{1i}}(v)$ is
defined by the relation
\[ A_i.v=p.h_i^{v}+r_{1ip}(v) \]
for a unique $r_{1ip}(v) \in P_p^3$. For any $v=(a_0,a_1, a_{2}) \in
P_p^3$, we define
\[ \begin{array}{l} d_1:=  A_1.v=(a_0,a_1, a_{2}) \\
d_2:=  A_2.v=(-a_1,-a_0+a_1,-a_1+a_2) \\ d_3:=
A_3.v=(a_0,-a_2,a_1-a_2). \end{array}\]

\vspace{3mm} Notice that if $v \in P_p^3$, then $d_1 \in P_p^3$.
Hence,  we have $h_1^v=0$ and $l_{\sigma_1}=0$. Therefore, we
obtain
\[ l_{\sigma_1}(u)=0 \quad \mbox{for} \quad u  \in \sigma_1= \langle v_0,v_1,v_3 \rangle .\]
So, by Remark \ref{explicacio} we deduce that for any $v \in
P_p^3$,
\[ D_v=\beta_1Z_1^-+\beta_2Z_{2}^-+\beta_3D_0 + \beta_4 D_1^+
+ \beta_5 D_1^-\] where \[ \beta_1=-l_{\sigma_2}(v_2) \quad
\beta_2=-l_{\sigma_3}(v_4), \quad \beta_3=-l_{\sigma_2}(w_0) \quad
\beta_4=-l_{\sigma_3}(w_1), \quad \mbox{and} \quad
\beta_5=-l_{\sigma_2}(w_2).
\] To determine these coefficients we will consider different
cases.

\vskip 2mm \noindent {\bf Case 1:} $a_0=a_1=a_2=0$.

In that case, $D_v=0$.

\vskip 2mm \noindent {\bf Case 2:} $a_1=a_2=0$ and $a_0 \neq 0$.

In that case, $d_2=(0,  -a_0,0)$ and $d_3=(a_0,0,0)$. Therefore,
$h_2^v=(0,-1,0)$, $h_3^v=0$, $l_{\sigma_2}=\hat{e}_0$,
$l_{\sigma_3}=0$ and thus $l_{\sigma_2}(v_2)=0$,
$l_{\sigma_2}(w_0)=-1$ and $l_{\sigma_2}(w_2)=0$ which gives us
\[ D_v=D_0. \]

\vskip 2mm \noindent {\bf Case 3:} $a_0=a_2=0$ and $a_1 \neq 0$.

In that case, $d_2=(-a_1, 0, -a_1)$ and $d_3=(0,0,a_1)$. Therefore,
$h_2^v=(-1,0,-1)$, $h_3^v=0$, $l_{\sigma_2}=\hat{e}_0+ \hat{e}_1$,
$l_{\sigma_3}=0$ and thus $l_{\sigma_2}(v_2)=-1$,
$l_{\sigma_2}(w_0)=0$ and $l_{\sigma_2}(w_2)=-1$ which gives us
\[ D_v=Z_1^-+D_1^-. \]

\vskip 2mm \noindent {\bf Case 4:} $a_0=a_1=0$ and $a_2 \neq 0$.

In that case, $d_2=(0,0,a_2)$ and $d_3=(0, -a_2, -a_2)$. Therefore,
$h_2^v=0$, $h_3^v=(0,-1,-1)$, $l_{\sigma_3}= \hat{e}_2$,
$l_{\sigma_2}=0$ and thus $l_{\sigma_3}(v_4)=-1$,
$l_{\sigma_3}(w_1)=-1$ which gives us
\[ D_v=Z_2^-+D_1^+. \]

\vskip 2mm \noindent {\bf Case 5:} $a_0,a_1 \neq 0$ and $a_2= 0$.

In that case, $d_2=(-a_1, -a_0+a_1, -a_1)$ and $d_3=(a_0,0,a_1)$. Therefore,
$h_2^v=(-1,0,-1)$ and $h_3^v=0$ if $-a_0+a_1 \geq 0$ or
$h_2^v=(-1,-1,-1)$ and $h_3^v=0$ if $-a_0+a_1 < 0$. The first case
do not contribute with a new summand and in the second case
$l_{\sigma_2}=2\hat{e}_0+ \hat{e}_1$ and $l_{\sigma_3}=0$. Thus
$l_{\sigma_2}(v_2)=-1$, $l_{\sigma_2}(w_0)=-1$ and
$l_{\sigma_2}(w_2)=-1$ which gives us
\[ D_v=Z_1^-+D_0+D_1^-. \]

\vskip 2mm \noindent {\bf Case 6:} $a_0,a_2 \neq 0$ and $a_1= 0$.

In that case, $d_2=(0, -a_0, a_2)$ and $d_3=(a_0,-a_2,-a_2)$.
Therefore, $h_2^v=(0,-1,0)$, $h_3^v=(0,-1,-1)$,
$l_{\sigma_2}=\hat{e}_0$ and  $l_{\sigma_3}=\hat{e}_2$. Thus
$l_{\sigma_2}(v_2)=0$, $l_{\sigma_2}(w_0)=-1$,
$l_{\sigma_2}(w_2)=0$, $l_{\sigma_3}(w_1)=-1$ and
$l_{\sigma_3}(v_4)=-1$ which gives us
\[ D_v=Z_2^-+D_0+D_1^+. \]

\vskip 2mm \noindent {\bf Case 7:} $a_1,a_2 \neq 0$ and $a_0= 0$.

In that case, $d_2=(-a_1, a_1, -a_1+a_2)$ and
$d_3=(0,-a_2,a_1-a_2)$. Therefore, the possibilities that we have
are \[ h_2^v=(-1,0,0), \quad h_3^v=(0,-1,-1) \quad \mbox{if} \quad
a_1-a_2 < 0 ; \]\[ h_2^v=(-1,0,0), \quad h_3^v=(0,-1,0) \quad
\mbox{if} \quad a_1-a_2 = 0 ; \]\[ h_2^v=(-1,0,-1), \quad
h_3^v=(0,-1,0) \quad \mbox{if} \quad a_1-a_2 > 0 . \] Arguing as
before the three news summands that we get are
\[ D_v=Z_1^-+Z_2^-+D_1^+, \quad D_v=Z_1^-+Z_2^- \quad D_v=Z_1^-+D_1^-+Z_2^-. \]

\vskip 2mm \noindent {\bf Case 8:} $a_0,a_1,a_2 \neq 0$.

In that case,  $d_2=(-a_1, -a_0+a_1, -a_1+a_2)$ and
$d_3=(a_0,-a_2,a_1-a_2)$.  Arguing as before, the three news summands that
we get are
\[ D_v=Z_1^-+Z_2^-+D_0+D_1^+, \quad D_v=Z_1^-+Z_2^- +D_0\quad D_v=Z_1^-+D_1^-+Z_2^-+D_0. \]

Putting all cases together we get that the different summands
appearing in (\ref{split}) for $d=3$ are {\footnotesize\[
\begin{array}{ll} \cT_3 \cong & \cO \oplus \cO(D_0) \oplus
\cO(Z_1^-+D_1^-) \oplus \cO(Z_1^-+D_1^-+D_0) \oplus
\cO(Z_2^-+D_1^+) \oplus \cO(Z_2^-+D_1^++D_0) \\ &  \oplus
\cO(Z_1^-+Z_2^-) \oplus \cO(Z_1^-+Z_2^-+D_1^+)  \oplus
\cO(Z_1^-+Z_2^-+D_1^-) \oplus \cO(Z_1^-+Z_2^-+D_0) \\ & \oplus
\cO(Z_1^-+Z_2^-+D_1^++D_0)  \oplus \cO(Z_1^-+Z_2^-+D_1^-+D_0)
\end{array}. \]}
and this concludes this initial case $d=3$.

 \vspace{3mm} For $3 < d=2l+1$ take $e_0, \ldots, e_{d-1}$ be a $\ZZ$-basis of the lattice $\ZZ^d$ with the
 convention that if $e_0, \ldots, e_{d-3}$ is a $\ZZ$-basis of $\ZZ^{d-2}$ we complete it to get
  $e_0, \ldots, e_{d-3}, e_{d-2}, e_{d-1}$ a $\ZZ$-basis of $\ZZ^{d}$.
  Recall that
 \[v_0=e_0,  \quad v_{2k-1}=e_k, \quad  v_{2k}=-e_k, \quad \mbox{for } 1 \leq k \leq d-1=2l \]
 \[w_0=e_1-e_0,  \quad w_{2j-1}=e_{2j-1}-e_{2j}, \quad  w_{2j}=e_{2j}-e_{2j-1}, \quad \mbox{for } 1 \leq j \leq l, \]
 are the ray generators of the fan $\Sigma_d$ associated to $X_d$.

 As we have seen in Remark \ref{explicacio}, to get all the
 different summands appearing in the splitting (\ref{split})
 it is enough to determine $l_{\sigma_1^d}$, $l_{\sigma_2^d}$ and
$l_{\sigma_3^d}$ where $\sigma_1^d$, $\sigma_2^d$ and $\sigma_3^d$
are three maximal cones of $\Sigma_d$ that altogether contain all
the ray generators $v_i$, $0 \leq i \leq 2d-2$ and $w_i$, $0 \leq
i \leq 2l$. We choose the following three maximal cones of
$\Sigma_d$:
\[ \begin{array}{l} \sigma_1^d:= \langle v_0,v_1,v_3, \cdots, v_{2(d-3)-1}, v_{2d-5}, v_{2d-3} \rangle,
 \quad \\  \sigma_2^d:= \langle v_2,v_6, \cdots, v_{2(d-4)}, w_0, w_2, \cdots , w_{2(l-1)}, v_{2(d-2)},  w_{2l} \rangle,
  \quad \\
\sigma_3^d:= \langle v_0,v_4, \cdots, v_{2(d-3)}, w_{1}, w_{3},
\cdots, w_{2l-3}, v_{2(d-1)}, w_{2l-1} \rangle. \end{array}\]
Notice that the set of ray generators of $\Sigma_d$ can be seen as
the set of ray generators of $\Sigma_{d-2}$ together with the ray
generators $v_{2d-5}, v_{2d-3}, v_{2(d-2)}, w_{2l}, v_{2(d-1)},
w_{2l-1} $ and that the following recursive relation holds:
\[\sigma_1^d=\langle  \sigma_1^{d-2} , v_{2d-5}, v_{2d-3} \rangle ,\]
\[\sigma_2^d= \langle  \sigma_2^{d-2}, v_{2(d-2)}, w_{2l} \rangle ,\]
\[\sigma_3^d= \langle  \sigma_3^{d-2} , v_{2(d-1)}, w_{2l-1} \rangle \]
where $\sigma_1^{d-2}, \sigma_2^{d-2}$ and $\sigma_3^{d-2}$ are
the corresponding maximal cones of $\Sigma_{d-2}$ that contain all
its ray generators.

Thus, the matrices $A_i^d$, $1 \leq i \leq 3$, having as the
$j$-th row the coordinates of the $j$-vector of $\sigma_i^d$
expressed in the basis $e_0, \cdots, e_{d-1}$ are:

\[A_1^d= \left ( \begin{array}{ccc}  A_1^{d-2} & \vline & 0 \\
\hline
 0 & \vline &
\begin{array}{ll} 1 & 0 \\ 0 & 1 \end{array}  \end{array} \right ), \quad
  A_2^d= \left ( \begin{array}{ccc} A_2^{d-2}   & \vline  & 0 \\

  \hline
  0 & \vline &
 \begin{array}{ll} 0 & -1 \\ 1 & -1 \end{array}  \end{array} \right ) , \]

\vskip 3mm \[A_3^d= \left ( \begin{array}{ccc}  A_3^{d-2} & \vline &0 \\ \hline  0
&\vline &
\begin{array}{ll} -1 & 0 \\ -1 & 1 \end{array}  \end{array} \right ) ,\]

\vskip 3mm \noindent
where the $A_i^{d-2}$, $1 \leq i \leq 3$, are the matrices having
as the $j$-th row the coordinates of the $j$-vector of
$\sigma_i^{d-2}$ expressed in the basis $e_0, \cdots, e_{d-3}$.
Their inverses are given by $B_1^d=A_1^d$,

\[B_2^d= \left ( \begin{array}{ccc}  B_2^{d-2} & \vline & 0 \\ \hline  0
&\vline &
\begin{array}{ll} -1 & 0 \\ -1 & 1 \end{array}  \end{array} \right ) \text{ and
}\quad B_3^d= \left ( \begin{array}{ccc}  B_3^{d-2} & \vline & 0
\\ \hline 0 & \vline &
\begin{array}{ll} -1 & 1 \\ -1 & 0 \end{array}  \end{array} \right ) \]

\vskip 3mm \noindent where the matrix $B_i^{d-2}$, $1 \leq i \leq 3$, is the inverse of
the matrix $A_i^{d-2}$.

 Fix the index $l=1$ corresponding to the cone $\sigma_1^d$. By (\ref{Cij}),
\[C_{1i}^d=(B_i^d)^{-1} B_1^d=A_i^d \]
and, as we pointed out before, if $\{U_{\sigma_j}, X_j^{u_j}
\}_{\sigma_j \in \Sigma_d}$ represents the zero divisor then for
any pair $i$, $j$,
\[u_{1i}=u_i-C_{1i}u_1=0. \]
Hence, for any $v^d \in P_p^d$, $h_i^{v^d}:=h_{1ip}^{u_{1i}}(v^d)$
is defined by the relation
\[ A_i^d.v^d=p.h_i^{v^d}+r_{1ip}(v^d) \]
for a unique $r_{1ip}(v^d) \in P_p^d$.

For any $v^d=(a_0, \cdots,a_{d-3},a_{d-2}, a_{d-1}) \in P_p^d$, we
define
\[ \begin{array}{l} d_1^d:=  A_1^d.v^d ,\\
d_2^d:=  A_2^d.v^d ,\\ d_3^d:=  A_3^d.v^d
.\end{array}\]
Notice that we can see $v^d$ as $v^d=(v^{d-2},a_{d-2}, a_{d-1})$
and hence we have
\[ \begin{array}{l} d_1^d=(d_1^{d-2}, a_{d-2},a_{d-1}), \\
d_2^d=(d_2^{d-2}, -a_{d-2},-a_{d-2}+a_{d-1}), \\
d_3^d=(d_3^{d-2}, -a_{d-1},a_{d-2}-a_{d-1}). \end{array} \]

\vspace{3mm}

 Notice that if $v^d \in P_p^d$, then $d_1^d \in P_p^d$. Hence,
we have $h_1^{v^d}=0$ and $l_{\sigma_1^d}=0$. Therefore, it only
remains to determine, for any $v^d \in P_p^d$, the functions
$l_{\sigma_2^d}$ and $l_{\sigma_3^d}$ and to write down the
corresponding $\cO_{X_d}(D_{v^d})$. To this end, we will proceed
by induction on odd $d$ and  we will consider four different
cases.

\vskip 2mm \noindent {\bf Case 1:} $a_{d-2}=a_{d-1}=0$.

In that case, for $i=2,3$ we have
\[ h_i^{v^d}=(h_i^{v^{d-2}},0,0).\]
Therefore, $l_{\sigma_3^d}=l_{\sigma_3^{d-2}}$, $l_{\sigma_3^d}=l_{\sigma_3^{d-2}}$ and using induction
the different summands that we get are \[
 \cT_3 \otimes \bigotimes_{k=2}^{l-1}\Big (\cO
\oplus \cO(Z_{2k}^-+D_k^+) \oplus \cO(Z_{2k-1}^-+D_k^-) \oplus
\cO(Z_{2k-1}^-+Z_{2k}^-)\]\[ \oplus \cO(Z_{2k-1}^-+Z_{2k}^-+D_k^-)
\oplus \cO(Z_{2k-1}^-+Z_{2k}^-+D_k^+)\Big ) \otimes \cO.
\]

\vskip 2mm \noindent {\bf Case 2:} $a_{d-2} \neq 0$ and $
a_{d-1}=0$.

In that case we have
\[ h_2^{v^d}=(h_2^{v^{d-2}},-1,-1), \text{ and }\]
\[ h_3^{v^d}=(h_3^{v^{d-2}},0,0).\]
Therefore, $l_{\sigma_2^d}=l_{\sigma_2^{d-2}}+\hat{e}_{d-2}$ and
$l_{\sigma_3^d}=l_{\sigma_3^{d-2}}$. Thus, using induction, the
different summands that we get in this case are \[
 \cT_3 \otimes \bigotimes_{k=2}^{l-1}\Big (\cO
\oplus \cO(Z_{2k}^-+D_k^+) \oplus \cO(Z_{2k-1}^-+D_k^-) \oplus
\cO(Z_{2k-1}^-+Z_{2k}^-) \]\[\oplus \cO(Z_{2k-1}^-+Z_{2k}^-+D_k^-)
\oplus \cO(Z_{2k-1}^-+Z_{2k}^-+D_k^+)\Big ) \otimes
\cO(Z_{2l-1}^-+D_l^-). \]

\vskip 2mm \noindent {\bf Case 3:} $a_{d-1} \neq 0$ and $
a_{d-2}=0$.

In that case we have
\[ h_2^{v^d}=(h_2^{v^{d-2}},0,0) \text{ and }\]
\[ h_3^{v^d}=(h_3^{v^{d-2}},-1,-1).\]
Therefore, $l_{\sigma_2^d}=l_{\sigma_2^{d-2}}$ and
$l_{\sigma_3^d}=l_{\sigma_3^{d-2}}+\hat{e}_{d-1}$. Thus, using
induction, the different summands that we get in this case are
\[
 \cT_3 \otimes \bigotimes_{k=2}^{l-1}\Big (\cO
\oplus \cO(Z_{2k}^-+D_k^+) \oplus \cO(Z_{2k-1}^-+D_k^-) \oplus
\cO(Z_{2k-1}^-+Z_{2k}^-) \]\[\oplus \cO(Z_{2k-1}^-+Z_{2k}^-+D_k^-)
\oplus \cO(Z_{2k-1}^-+Z_{2k}^-+D_k^+)\Big ) \otimes
\cO(Z_{2l}^-+D_l^+) .\]

\vskip 2mm \noindent {\bf Case 4:} $a_{d-1} \neq 0$ and $ a_{d-2}
\neq 0$.

In that case the following possibilities can occur:
\[ \begin{array}{llll} 4.1) & h_2^{v^d}=(h_2^{v^{d-2}},-1,-1) &
\mbox{and} &  h_3^{v^d}=(h_3^{v^{d-2}},-1,0); \\
4.2) & h_2^{v^d}=(h_2^{v^{d-2}},-1,0) &
\mbox{and} &  h_3^{v^d}=(h_3^{v^{d-2}},-1,0); \\
4.3) & h_2^{v^d}=(h_2^{v^{d-2}},-1,0) &
\mbox{and} &   h_3^{v^d}=(h_3^{v^{d-2}},-1,-1). \\
\end{array} \]
If $4.1)$ occurs, then $l_{\sigma_2^d}=l_{\sigma_2^{d-2}}+
\hat{e}_{d-2}$ and
$l_{\sigma_3^d}=l_{\sigma_3^{d-2}}+\hat{e}_{d-2}+\hat{e}_{d-1}$
and the different summands that we get in this case are
\[\cT_3\otimes
\bigotimes_{k=2}^{l-1}(\cO
\oplus \cO(Z_{2k}^-+D_k^+) \oplus \cO(Z_{2k-1}^-+D_k^-) \oplus
\cO(Z_{2k-1}^-+Z_{2k}^-) \oplus \] $$\cO(Z_{2k-1}^-+Z_{2k}^-+D_k^-)
\oplus \cO(Z_{2k-1}^-+Z_{2k}^-+D_k^+)\Big )  \otimes
\cO(Z_{2l-1}^-+Z_{2l}^-+D_l^-)   .$$

If $4.2)$ occurs, then $l_{\sigma_2^d}=l_{\sigma_2^{d-2}}+
\hat{e}_{d-2}+\hat{e}_{d-1}$ and
$l_{\sigma_3^d}=l_{\sigma_3^{d-2}}+\hat{e}_{d-2}+\hat{e}_{d-1}$
and the different summands that we get in this case are
\[
 \cT_3 \otimes \bigotimes_{k=2}^{l-1}\Big (\cO
\oplus \cO(Z_{2k}^-+D_k^+) \oplus \cO(Z_{2k-1}^-+D_k^-) \oplus
\cO(Z_{2k-1}^-+Z_{2k}^-) \oplus \]\[ \cO(Z_{2k-1}^-+Z_{2k}^-+D_k^-)
\oplus \cO(Z_{2k-1}^-+Z_{2k}^-+D_k^+)\Big ) \otimes
\cO(Z_{2l}^-+Z_{2l-1}^-).\]

Finally, if $4.3)$ occurs, then
$l_{\sigma_2^d}=l_{\sigma_2^{d-2}}+ \hat{e}_{d-2}+\hat{e}_{d-1}$
and $l_{\sigma_3^d}=l_{\sigma_3^{d-2}}+\hat{e}_{d-1}$ and the
different summands that we get in this case are \[
 \cT_3 \otimes  \bigotimes_{k=2}^{l-1}\Big (\cO
\oplus \cO(Z_{2k}^-+D_k^+) \oplus \cO(Z_{2k-1}^-+D_k^-) \oplus
\cO(Z_{2k-1}^-+Z_{2k}^-) \oplus\]\[  \cO(Z_{2k-1}^-+Z_{2k}^-+D_k^-)
\oplus \cO(Z_{2k-1}^-+Z_{2k}^-+D_k^+)\Big )  \otimes
\cO(Z_{2l}^-+Z_{2l-1}^-+D_l^+).  \]

\vspace{3mm}
 Putting together
the four cases we obtain that the different summands appearing in the splitting
(\ref{split}) are precisely \[
 \cT_3 \otimes   \bigotimes_{k=2}^{l}\Big (\cO
\oplus \cO(Z_{2k}^-+D_k^+) \oplus \cO(Z_{2k-1}^-+D_k^-) \oplus
\cO(Z_{2k-1}^-+Z_{2k}^-) \]\[ \oplus \cO(Z_{2k-1}^-+Z_{2k}^-+D_k^-)
\oplus \cO(Z_{2k-1}^-+Z_{2k}^-+D_k^+)\Big )
\]
which proves what we want.
\end{proof}



\section{Orthogonal basis}

This section contains the main theorem of this work and it has as
a main goal to construct a full strongly exceptional collection of
line bundles in the derived category $D^b(X)$, that is an
orthogonal basis made up of line bundles, where $X$ is a smooth
Fano toric variety with (almost) maximal Picard number.  We will
start recalling the notions of exceptional sheaves, exceptional
collections of sheaves, strongly exceptional collections of
sheaves and full strongly exceptional collections of sheaves as
well as the facts on derived categories needed in the rest of the
paper.

\begin{definition}\label{exceptcoll}
Let $Y$ be a smooth projective variety.

(i) A  coherent sheaf $F$ on  $Y$ is {\em exceptional} if $\Hom
(F,F)=K $ and $\Ext^{i}_{\cO _Y}(F,F)=0$ for $i>0$,

(ii) An ordered collection $(F_0,F_1,\ldots ,F_m)$  of coherent
sheaves on $Y$ is an {\em exceptional collection} if each sheaf
$F_{i}$ is exceptional and $\Ext^i_{\cO _Y}(F_{k},F_{j})=0$ for
$j<k$ and $i \geq 0$.

(iii) An exceptional collection $(F_0,F_1,\ldots ,F_m)$ is a {\em
strongly exceptional collection} if in addition $\Ext^{i}_{\cO
_Y}(F_j,F_k)=0$ for $i\ge 1$ and  $j \leq k$.

(iv) An ordered collection $(F_0,F_1,\cdots ,F_m)$ of coherent
sheaves on $Y$ is a {\em full (strongly) exceptional collection}
if it is a (strongly) exceptional collection  and $F_0$, $F_1$,
$\cdots $ , $F_m$ generate the bounded derived category $D^b(Y)$.
\end{definition}

\begin{remark} \label{length} The existence of a full strongly
exceptional collection $(F_0,F_1,\cdots,F_m)$ of coherent sheaves
on a smooth projective variety $Y$ imposes a rather  strong
restriction on $Y$, namely that the Grothendieck group
$K_0(Y)=K_0(\cO _Y-mod)$ is isomorphic to $\ZZ^{m+1}$.
\end{remark}

\vspace{3mm}

It is natural to ask whether $D^b(Y)$ is freely and finitely
generated. More precisely, we are lead to consider the following
problem

\begin{problem}\label{prob1}
To characterize smooth projective varieties $Y$ which have a full
strongly exceptional collection of coherent sheaves and, even
more, if there is one made up of line bundles.
\end{problem}

This problem is far from being solved and in this paper we will
restrict our attention to the particular case of toric varieties.
Toric varieties admit a combinatorial description which allows
many invariants to be expressed
 in terms of combinatorial data. We will use this fact
 to describe the derived category of
smooth Fano toric varieties with (almost) maximal Picard number
and, in particular, we will give positive contributions to the
above problem and to the following Conjecture:

 \begin{conjecture} \label{conj11} Every smooth complete
 Fano toric  variety $X$ has a full strongly exceptional
collection of line bundles.
\end{conjecture}

So far, only partial results are known but there are some numerical evidences towards
Conjecture \ref{conj1} (For detailed information about Conjecture \ref{conj1}, the reader can consult \cite{CDMR}, \cite{CMZ}, \cite{CM} and the references quoted there).

\vspace{3mm} Let us start dealing with smooth Fano $d$-dimensional
toric varieties with Picard number $2d$ or $2d-1$ which are
products of toric varieties of smaller dimension. In that
case, we will use the following result:

\vspace{3mm}

\begin{proposition} \label{prod1} Let $X_1$ and $X_2$ be two smooth projective
varieties and let $(F_0^{i},F_1^{i},\ldots ,F_{n_{i}}^{i})$ be   a
full strongly exceptional collection of locally free sheaves on
$X_i$, $i=1,2$.  Then, $$(F_0^{1}\boxtimes
F_0^{2},F_1^{1}\boxtimes F_0^{2},\ldots ,F_{n_1}^{1}\boxtimes
F_0^{2},  \ldots , F_0^{1}\boxtimes F_{n_2}^{2},F_1^{1}\boxtimes
F_{n_2}^{2},\ldots ,F_{n_1}^{1}\boxtimes F_{n_2}^{2})$$ is a full
strongly exceptional collection of locally free sheaves on $X_1
\times X_2$.
\end{proposition}
\begin{proof}
See \cite{CMZ}; Proposition 4.16
\end{proof}

Applying this result we get

\begin{proposition}
\label{casproducte} Let $X$ be a $d$-dimensional smooth Fano toric
variety which is isomorphic to either $(S_3)^{\frac{d}{2}}$  or
$S_2 \times (S_3)^{\frac{d-2}{2}}$ if $d$ is even, and  $\PP^1
\times (S_3)^{\frac{d-1}{2}}$ if $d$ is odd. Then, $X$ has a full
strongly exceptional collection made up of line bundles.
\end{proposition}
\begin{proof}
It is well known that $\PP^1$ has a full strongly exceptional
collection made up of line bundles. On the other hand, by
\cite{CMZ}; Proposition 4.19, $S_2$ and $S_3$ both have a full
strongly exceptional collection of line bundles. Thus, we can
conclude by applying reiteratively Proposition \ref{prod1}.
\end{proof}

\vspace{3mm}

Now we will deal with the remaining case of a $d$-dimensional
smooth Fano toric variety $X$ with Picard number $2d-1 \leq \rho_X
\leq 2d$, namely  $d$ will be odd and $X$  isomorphic to $X_d$: a
toric $(S_3)^{\frac{d-1}{2}}$-fiber bundle over $\PP^1$. In this
case, we will apply Bondal's criterium. Roughly speaking, this criterium
 asserts  that,
under certain restrictions, the different summands of the
splitting of the Frobenius direct image $(\pi_p)_*(\cO_{X_d})$
of the tautological line bundle can be ordered in such a way that they
 form a full strongly
exceptional collection of line bundles.
 We are going to recall it
after fixing some notation.

\vspace{3mm}

\begin{notation}
For any irreducible toric  curve $C$ in an $n$-dimensional toric
variety $X$, we denote by $D_1^C, \cdots, D_{n-1}^C$ the irreducible
toric divisors containing $C$ and we denote by $(a_1^C, \cdots,
a_{n-1}^C)$ the corresponding intersections numbers of the divisors
$D_i^C$ with $C$.
\end{notation}

\begin{proposition} {\bf (Bondal's criterium)}
\label{Bondalcriteri} Let $X$ be a smooth $n$-dimensional toric variety.
Assume that for any  irreducible toric  curve $C$ on $X$, the
coefficients $a_i^C$ verify $a_i^C \geq -1$ for $1 \leq i \leq
n-1$ and that no more than one is equal to $-1$. Then, for $p \gg
0$, a suitable order of  the different summands of $(\pi_p)_*(\cO_X)^{\vee}$ form a full
strongly exceptional collection of line bundles on $X$.
\end{proposition}
\begin{proof} See \cite{bondal}.
\end{proof}

\vspace{3mm}
\begin{remark}
\label{coef} Let $X$ be a smooth toric variety of dimension $d$ and let $C$ be an irreducible toric curve. Let us compute the numerical invariants  $a_{i}^C$. To this end,
we consider $u_1, \cdots, u_{n-1}$ the generators of the cone corresponding
to $C$ and let $u_+$, $u_-$ be the additional generators of the
two maximal  cones adjacent to it. Then, there is a
relation
\[u_+ + u_- + \sum_{i=1}^{n-1}a_i^Cu_i=0 \]
in which the coefficients $a_i^C$ are the required intersection
numbers.
\end{remark}

\vspace{3mm}

\begin{theorem}
\label{casBondal} Let $X_d$ be a toric
$(S_3)^{\frac{d-1}{2}}$-fiber bundle over $\PP^1$. Then,  a suitable order of the
summands of  {\footnotesize\[
 \cT_3 \otimes \bigotimes_{k=2}^{l}(\cO
\oplus \cO(Z_{2k}^-+D_k^+) \oplus \cO(Z_{2k-1}^-+D_k^-) \oplus
\cO(Z_{2k-1}^-+Z_{2k}^-) \oplus \cO(Z_{2k-1}^-+Z_{2k}^-+D_k^-)
\oplus \cO(Z_{2k-1}^-+Z_{2k}^-+D_k^+))
\]}
where {\footnotesize\[ \begin{array}{ll} \cT_3 \cong & \cO \oplus
\cO(D_0) \oplus \cO(Z_1^-+D_1^-) \oplus \cO(Z_1^-+D_1^-+D_0)
\oplus \cO(Z_2^-+D_1^+) \oplus \cO(Z_2^-+D_1^++D_0) \\ &  \oplus
\cO(Z_1^-+Z_2^-) \oplus \cO(Z_1^-+Z_2^-+D_1^+)  \oplus
\cO(Z_1^-+Z_2^-+D_1^-) \oplus \cO(Z_1^-+Z_2^-+D_0) \\ & \oplus
\cO(Z_1^-+Z_2^-+D_1^++D_0)  \oplus \cO(Z_1^-+Z_2^-+D_1^-+D_0)
\end{array} \] }
form a full strongly exceptional collection of line bundles on
$X_d$.
\end{theorem}
\begin{proof} First of all notice that we have exactly $2 \cdot
6^{\frac{d-1}{2}}$ summands which by (\ref{rankK0}) is the rank of
the Grothendieck group $K_0(X_d)$. Hence, the cardinality of any full strongly exceptional collection on $X_d$ is $2 \cdot
6^{\frac{d-1}{2}}$ (see Remark \ref{length}).

\vspace{3mm}

By Theorem \ref{summandspliting} and Proposition
\ref{Bondalcriteri}, we will conclude if we prove that $X_d$
verifies Bondal's criterium. To this end, let $e_0, \ldots,
e_{d-1}$ be a $\ZZ$-basis of the lattice $\ZZ^d$ and  denote by
 \begin{equation} \label{base1} v_0=e_0,  \quad v_{2k-1}=e_k, \quad  v_{2k}=-e_k, \quad \mbox{for } 1 \leq k \leq d-1=2l
 \end{equation}
 \begin{equation} \label{base2} w_0=e_1-e_0,  \quad w_{2j-1}=e_{2j-1}-e_{2j}, \quad  w_{2j}=e_{2j}-e_{2j-1}, \quad \mbox{for } 1 \leq j \leq l,
 \end{equation}
 the ray generators of the fan $\Sigma_d$ associated to $X_d$. We will proceed by induction on odd $d$.

Let $d=3$. In the following table, we write down the
two-dimensional cones $\sigma_C$ associated to any
irreducible toric curve $C\subset X_d$ and the
 additional generators $u_+$, $u_ -$ of the
two $3$-dimensional cones adjacent to it.

\vspace{5mm}

\begin{center}

\begin{tabular}{|c|c|c|c|c|c|c|c|c|}
 \cline{1-4} \cline{6-9}
   & $\sigma_C$ & $u_+$ & $u_ -$  &  \hspace{15mm} &  & $\sigma_C$ & $u_+$ & $u_ -$   \\
  \cline{1-4} \cline{6-9}
  1 & $\langle v_1,v_3 \rangle $ & $v_0$ & $w_0$ & \hspace{15mm} &  10 & $\langle v_3, w_0 \rangle $ & $v_1$ & $w_2$ \\
  \cline{1-4} \cline{6-9}
  2 & $\langle v_1,  v_0 \rangle $ &$ v_3$ & $w_1$ & \hspace{15mm} &  11 & $\langle v_3,w_2 \rangle $ & $v_0$ & $w_0$\\
  \cline{1-4} \cline{6-9}
  3 & $\langle v_1,w_0 \rangle $ & $v_3$ & $w_1$ & \hspace{15mm} &  12 & $\langle v_4,v_0 \rangle $ & $v_2$ & $w_1$ \\
  \cline{1-4} \cline{6-9}
  4 & $\langle v_1,w_1 \rangle $& $v_0$ & $w_0$ & \hspace{15mm} &  13 & $\langle v_4,w_0 \rangle $ & $v_2$ & $w_1$\\
  \cline{1-4} \cline{6-9}
  5 & $\langle v_2,v_4 \rangle $ & $v_0$ & $w_0$ & \hspace{15mm} &  14 & $\langle v_4,w_1 \rangle $ & $v_0$ & $w_0$\\
  \cline{1-4} \cline{6-9}
  6 & $ \langle v_2,v_0 \rangle $ & $v_4$ & $w_2$ & \hspace{15mm} &  15 & $\langle v_0,w_1 \rangle $ & $v_1$ & $v_4$\\
  \cline{1-4} \cline{6-9}
  7 & $ \langle v_2,w_0 \rangle $ & $v_4$ & $w_2$ & \hspace{15mm} &  16 & $\langle v_0,w_2 \rangle $ & $v_2$ & $v_3$\\
  \cline{1-4} \cline{6-9}
   8 & $ \langle v_2,w_2 \rangle $ & $v_0$ & $w_0$ & \hspace{15mm} &  17 & $\langle w_1,w_0 \rangle $ & $v_1$ & $v_4$\\
  \cline{1-4} \cline{6-9}
  9 & $ \langle v_3,w_0 \rangle $ & $v_1$ & $w_2$ & \hspace{15mm} &  18 & $\langle w_2,w_0 \rangle $ & $v_2$ & $v_3$\\
  \cline{1-4} \cline{6-9}
\end{tabular}
\end{center}

\vskip 4mm

By Remark \ref{coef}, we have to check that for any
irreducible toric curve $C$ in the above table the coefficients of the
relation
\begin{equation} \label{equ} u_+ + u_- +
a_1^Cu_1+a_2^Cu_2=0 \end{equation} being $u_1,u_2$  the ray
generators of $\sigma_C$,  are greater or equal to $-1$ and
at most there is one equal to $-1$.

 Consider the first case. If
 \[  v_0  + w_0 +a_1^Cv_1 +a_2^Cv_3 = 0 \]
 then by (\ref{base1}) and (\ref{base2}), we must have $a_1^C=-1$ and
 $a_2^C=0$. Hence the condition is verified. The remaining cases
 can also be checked by direct computation, and we left the details to the reader.

\vspace{3mm}

Let $d>3$ be an odd integer. Let $C$ be any  irreducible toric
curve, denote by $u_1, \cdots, u_{n-1}$  the generators of the
cone $\sigma_C$ corresponding to $C$ and let $u_+$, $u_-$ be the additional
generators of the two maximal  cones adjacent to it. Then,
there is a unique relation \begin{equation} \label{equ} u_+ + u_- +
\sum_{i=1}^{n-1}a_i^Cu_i=0 \end{equation}  and by Remark
\ref{coef} we have to prove that all the coefficients $a_i^C$ of
this relation are greater or equal to $-1$ and  at most there
is one equal to $-1$.

First of all notice that $\sigma_+:= \langle u_+,u_1,\cdots,
u_{d-1} \rangle$ is a maximal cone in $\Sigma_d$. Thus, it
contains at least two vectors, $z_1, z_2$ belonging to the set
\begin{equation} S:= \{ \pm e_{d-2}, \pm e_{d-1}, \pm ( e_{d-2}-e_{d-1})
\}.\end{equation} But it follows from (\ref{primitive}) that this
set does not contain three vectors defining a 3-dimensional cone
in $\Sigma_d$. Thus, $\sigma_+$ contains exactly two vectors
$z_1$, $z_2$ belonging to the set $S$ and moreover, the only possibilities for the pair $(z_1,z_2)$
are
\begin{equation}
\label{parelles}
 \begin{array}{ll}
(e_{d-2},e_{d-1}),  & (-e_{d-2},-e_{d-1}), \\
(e_{d-2},e_{d-2}-e_{d-1}),  & (-e_{d-2},e_{d-1}-e_{d-2}), \\
(e_{d-1},-e_{d-2}+e_{d-1}),  & (-e_{d-1},e_{d-2}-e_{d-1}) \\
\end{array}\end{equation}
because,  by
(\ref{primitive}), any other pair is a primitive collection.
The same argument shows that $\sigma_-:= \langle u_-,u_1,\cdots,
u_{d-1} \rangle$ contains exactly two vectors $z_1', z_2' $
belonging to the set $S$ and the only possibilities for the pair
$(z_1',z_2')$ are the ones in (\ref{parelles}).

From the definition is clear that the sets $\{z_1,z_2 \}$ and
$\{z_1',z_2 '\}$ coincide or they have at least one vector in
common. Keeping in mind this remark we will distinguish two cases.

{\bf Case 1:} $\{z_1,z_2 \}=\{z_1',z_2 '\}$

In that case, necessarily $z_1,z_2 \in \{u_1, \cdots, u_{d-1} \}$.
Renumbering if necessary, we can assume that $z_1=u_{d-2}$ and
$z_2=u_{d-1}$. Since the relation $(\ref{equ})$ must be verified, $z_1$ has to be canceled against
$z_2$. But $(z_1,z_2)$ is one of
the pairs (\ref{parelles}). Therefore, the only possibility is
$a_{d-2}^C=a_{d-1}^C=0$ and the relation (\ref{equ}) turns to be
\[u_+ + u_- + \sum_{i=1}^{d-3}a_i^Cu_i=0 .\] Hence, by hypothesis
of induction, the coefficients $a_i^C$  are greater or equal to -1  and at most one is equal to -1.

{\bf Case 2:} $\{z_1,z_2 \}$ and $\{z_1',z_2 '\}$ have one vector
in common.

In that case, renumbering, if necessary, the only possibility is
$z_1=u_{d-1}=z_1'$, $z_2=u_+$ and $z_2'=u_ -$. Thus, we must have $z_2+z_2'+a_{d-1}^Cz_1=u_++u_-+ a_{d-1}^Cu_{d-1}=0$. Therefore,  the only possibility is
$a_{d-1}^C=0$, $1$ or $-1$ and we get the relation
\[ \sum_{i=1}^{d-2}a_i^Cu_i=0. \]
If in this relation there is one $u_i$ of type $u_i=e_j-e_{j-1}$,
then it should be canceled with $-e_j$ and $e_{j-1}$; or  $e_j$
and $-e_{j-1}$; or $e_j$ and $e_{j-1}$; or $-e_j$ and $-e_{j-1}$.
But according to the list of primitive collections
(\ref{primitive}), none of this possibilities can occur since none
of them define a 3-dimensional cone. So, the relation only
contains vectors $u_i$ of type $\pm e_i$ and thus, the only possibility is
$a_i^C=0$ for all $1 \leq i \leq d-2$.

Therefore, for any  irreducible toric  curve $C\subset X_d$, Bondal's condition
is verified and therefore, the collection can be ordered in such a way that we get a full
strongly exceptional collection of line bundles.
\end{proof}
 \vspace{3mm}

Summing up we get our main result:

\begin{theorem} \label{mainthm}
Let $X$ be a smooth Fano $d$-dimensional toric variety with Picard
number $\rho_X$ with $2d-1 \leq \rho_X \leq 2d$. Then, $X$ has a
full strongly exceptional collection of line bundles.
\end{theorem}
\begin{proof}
It follows from Proposition \ref{classificacioVarietat},
Proposition \ref{casproducte} and Proposition \ref{casBondal}.
\end{proof}

\end{document}